\documentclass[12pt]{amsart}

\usepackage{amstext}
\usepackage{amsfonts}
\usepackage{amsthm}
\usepackage{amsmath}
\usepackage{amssymb}
\usepackage{latexsym}
\usepackage{amsfonts}
\usepackage{graphicx}
\usepackage{texdraw}
\usepackage{graphpap}
\usepackage{subfigure}
\usepackage{enumerate}

\usepackage{mathtools}
\mathtoolsset{showonlyrefs,showmanualtags}

\usepackage[pagebackref,hypertexnames=false, colorlinks, citecolor=red, linkcolor=red]{hyperref} 
\usepackage[backrefs]{amsrefs}

\input txdtools

\makeatletter
\newcommand*{\defeq}{\mathrel{\rlap{%
                     \raisebox{0.3ex}{$\m@th\cdot$}}%
                     \raisebox{-0.3ex}{$\m@th\cdot$}}%
                     =}
										
\newcommand*{\eqdef}{=
										 \mathrel{\rlap{%
                     \raisebox{0.3ex}{$\m@th\cdot$}}%
                     \raisebox{-0.3ex}{$\m@th\cdot$}}%
										}
\makeatother

\usepackage{bbm}   
\newcommand{\unit}{\mathbbm 1} 

\newcommand{\R}{\mathbb{R}}

\newcommand{\mcd}{\mathcal{D}}
\newcommand{\mcr}{\mathcal{R}}

\newcommand{\ep}{\epsilon}
\newcommand{\lb}{\lambda}
\newcommand{\mbs}{\mathbb{S}}

\newcommand{\wbmod}{BMO_{\mcd}(w)}
\newcommand{\wbmosd}{BMO^2_{\mcd}(w)}

\newcommand{\nbmosd}{BMO^2_{\mcd}(\nu)}
\newcommand{\avgb}{\left<b\right>}

\newcommand{\La}{\left\langle }
\newcommand{\Ra}{\right\rangle }

{\end{list}}

\numberwithin{equation}{section}

\newtheorem{thm}[equation]{Theorem}
\newtheorem{lm}[equation]{Lemma}
\newtheorem{cor}[equation]{Corollary}

\newtheorem{prop}[equation]{Proposition}
\newtheorem*{prop*}{Proposition}

\theoremstyle{remark}

\newtheorem*{rem*}{Remark}

\begin{document}

\title{Commutators in the Two-Weight Setting}

\author{Irina Holmes}
\address{Irina Holmes, School of Mathematics\\ Georgia Institute of Technology\\ 686 Cherry Street\\ Atlanta, GA USA 30332-0160}
\email{irina.holmes@math.gatech.edu}

\author[Michael T. Lacey]{Michael T. Lacey$^{\dagger}$}
\address{Michael T. Lacey, School of Mathematics\\ Georgia Institute of Technology\\ 686 Cherry Street\\ Atlanta, GA USA 30332-0160}
\email{lacey@math.gatech.edu}
\thanks{$\dagger$  Research supported in part by a National Science Foundation DMS grant \#1265570.}

\author[Brett D. Wick]{Brett D. Wick$^{\ddagger}$}
\address{Brett D. Wick, School of Mathematics\\ Georgia Institute of Technology\\ 686 Cherry Street\\ Atlanta, GA USA 30332-0160}
\email{wick@math.gatech.edu}
\thanks{$\ddagger$  Research supported in part by National Science Foundation DMS grants \#1603246 and \#1560955.}

\subjclass[2000]{Primary: 42, 42A, 42B, 42B20, 42B25, 42A50, 42A40, }
\keywords{Commutators, Calder\'on--Zygmund Operators, Bounded Mean Oscillation, Weights}

\begin{abstract}
Let $  R $ be the vector of Riesz transforms on $ \mathbb R ^{n}$, and 
 let  $\mu,\lambda\in A_p$ be two weights on  $ \mathbb R ^{n}$, $ 1< p < \infty $.  
 The two-weight norm inequality for the commutator
 $ \lVert [b, R] \;:\; L ^{p} (\mu ) \mapsto L ^{p} (\lambda )\rVert$  is shown to be 
equivalent to the function $ b$ being in a $BMO$ space adapted to $ \mu $ and $ \lambda $. 
This is a common extension of a result of Coifman-Rochberg-Weiss in the case of both $ \lambda $ and $ \mu $ 
being Lebesgue measure,  and Bloom in the case of dimension one. 
\end{abstract}

\maketitle
\setcounter{tocdepth}{1}
\tableofcontents

\section{Introduction and Statement of Main Results}

The foundational paper of Coifman-Rochberg-Weiss \cite{CRW} set out a real-variable counterpart to 
a classical theorem of Nehari \cite{MR0082945}.  
It characterized $BMO$, the real-variable space of functions with bounded mean oscillation,  in terms of commutators with Riesz transforms.  Several lines of investigation came out of this work:  generalizations to spaces of homogeneous type, \cites{MR618077}, multi-parameter extensions, \cites{DalencOu,MR1961195,MR2530853,MR2491875}, connections to factorization of function spaces, \cites{MR957051,MR518108, MR2663410, MR2898705}, div-curl lemmas, \cites{CLMS, MR3007645} and additional interpretations in operator theory \cites{MR0467384,MR1880830,MR0082945}. 

In 1985,  Bloom \cite{Bloom} proved a \emph{two-weight} extension 
of the Nehari \cite{MR0082945} result in one dimension. 
In particular, for the Hilbert transform, 
$$
Hf(x):=\textnormal{p.v.} \frac{1}{\pi}\int_{\mathbb{R}} \frac{f(y)}{x-y}\,dy, 
$$
a choice of $ 1< p < \infty $, and two weights $\mu$ and $\lambda$ in Muckenhoupt's $ A_p $ class, (see Section \ref{ss:apweights} for the definitions of these weights), the commutator 
$[b,H](f)=bHf-H(bf)$ is bounded from $ L ^{p} (\mu )$ to $ L ^{p} (\nu ) $ if and only if the function $ b$ satisfies 
$$
\left\Vert b\right\Vert_{BMO(\nu)}:=\sup_{Q} \left( \frac{\int_Q |b - \left<b\right>_Q|\,dx}{\int_{Q} \nu\,dx} \right)<\infty, 
$$
where $ \nu=\mu^{\frac{1}{p}}\lambda^{-\frac{1}{p}}$.  If $ \mu = \lambda $, the $ BMO (\nu )$ space 
is the classical one, and the result is well-known. But, in full generality, this is a subtle result, as it is 
a characterization in the triple of $ \mu , \lambda $ and $ b$.  

The purpose of this paper is to extend Bloom's result to the setting of Coifman-Rochberg-Weiss.   
Recall that a Calder\'on--Zygmund operator associated to a kernel $K(x,y)$ is an integral operator:
$$
Tf(x):=\int_{\mathbb{R}^n} K(x,y)f(y)\,dy,\quad x\notin\textnormal{supp} f,
$$
and that the kernel satisfies the standard size and smoothness estimates
\begin{gather*}
\left\vert K(x,y)\right\vert  \leq  \frac{C}{\left\vert x-y\right\vert^n}, \\
\left\vert K(x+h,y)-K(x,y)\right\vert +\left\vert K(x,y+h)-K(x,y)\right\vert  \leq  C\frac{\left\vert h\right\vert^{\delta}}{\left\vert x-y\right\vert^{n+\delta}},
\end{gather*}
for all $\left\vert x-y\right\vert>2\left\vert h\right\vert>0$ and a fixed $\delta\in (0,1]$.

Our first main result is the following upper bound for the commutator, $[b,T](f):=bTf-T(bf)$, with a Calder\'on--Zygmund operator.
\begin{thm} \label{T:UpperBound}
Let $T$ be a Calder{\'o}n-Zygmund operator on $\R^n$ and $\mu, \lb \in A_p$ with $1<p<\infty$. Suppose $b \in BMO(\nu)$, where $\nu = \mu^{\frac{1}{p}} \lb^{-\frac{1}{p}}$. Then
	\begin{equation*} 
	\| [b, T]: L^p(\mu) \rightarrow L^p(\lb) \| \leq c \|b\|_{BMO(\nu)},
	\end{equation*}
where $c$ is a constant depending on the dimension $n$, the operator $T$, and $\mu$, $\lb$, and $p$.
\end{thm}

Recall that the Riesz transforms are defined by:
$$
R_j(f)(x):=\textnormal{p.v.}\frac{\Gamma\left(\frac{n+1}{2}\right)}{\pi^{\frac{n+1}{2}}}\int_{\mathbb{R}^n} f(y) \frac{x_j-y_j}{\left\vert x-y\right\vert^{n+1}}\,dy,\quad j=1,\ldots, n.
$$
Specializing to the Riesz transforms, we are able to characterize $BMO(\nu)$ in terms of the boundedness of the commutators.  This gives a joint generalization of  Bloom and Coifman-Rochberg-Weiss, which is the main result of the paper.

\begin{thm} \label{T:LowerBound}
For $1<p<\infty$, and  $\mu,\lb\in A_p$, set $\nu = \mu^{\frac{1}{p}}\lb^{-\frac{1}{p}}$. 
Then there are constants $ 0< c < C < \infty $, depending only on $ n, p, \mu$ and $ \lambda $, for which 
\begin{equation}\label{e:Requivalence}
c\|b\|_{BMO(\nu)} \leq \sum_{i=1}^n \left\| [b, R_i] : L^p(\mu) \rightarrow L^p(\lb)\right\|\leq C\|b\|_{BMO(\nu)}.
\end{equation}
\end{thm}

In general two-weight inequalities are challenging, with complete characterizations for operators more difficult to obtain.  In the case of positive operators a complete characterization has been obtained by Sawyer in \cites{MR930072, MR676801}.  For operators with a more singular nature we point to the work of Nazarov, Treil and Volberg in \cite{MR2407233}, Lacey, Sawyer, Shen, and Uriarte-Tuero in \cite{MR3285857}, Lacey in \cite{MR3285858}, Lacey, Sawyer, Shen, Uriarte-Tuero and Wick in \cite{combined} and Lacey and Wick in \cite{LW}.  This points to further novelty in the main result since it obtains a characterization in terms of the triple $(b,\mu,\lambda)$ via a special $BMO$ space.

Similar to \cite{CRW}, the equivalence in Theorem \ref{T:LowerBound} yields a weak-factorization result for weighted Hardy spaces.  

\begin{cor} \label{C:Factorization} Under the hypotheses and notation  of Theorems \ref{T:UpperBound} and \ref{T:LowerBound},  let $\lb' \defeq \lb^{1-q}$ and $ T$ be a Calder{\'o}n-Zygmund operator on $\R^n$.  We have the inequality 
$$\| g_1 (Tg_2) - (T^*g_1)g_2\|_{H^1(\nu)} \leq c \|g_1\|_{L^q(\lb')} \|g_2\|_{L^p(\mu)},$$
where $c$ is a constant depending on the dimension, the operator $T$, and on $\mu$, $\lb$, and $p$.
Conversely, there exists a constant $c$ so that every $f \in H^1(\nu)$ can be written as
	\begin{equation} \label{E:Decomposition}
	f(x) = \sum_{i=1}^n \sum_{j=1}^{\infty}  g_j^i(x) R_i h_j^i(x) + h_j^i(x) R_i g_j^i(x),
	\end{equation}
where $R_i$ is the Riesz transform in the $i$th variable, and $g_j^i \in L^q(\lb')$, $h_j^i \in L^p(\mu)$ with
	$$\sum_{i=1}^{n}\sum_{j=1}^{\infty} \|g_j^i\|_{L^q(\lb')} \|h_j^i\|_{L^p(\mu)} \leq c\|f\|_{H^1(\nu)}.$$
\end{cor}

In the special case of the Hilbert transform, and $ p=2$, the paper \cite{HLW} gives a  `modern' proof of Bloom's result. 
We follow the outlines of that proof in the current setting, which has its genesis in \cite{MR1756958}.  Using a Haar shift representation of $ T$, 
a commutator can written out as a sum of several terms. Most of these are paraproducts,  with symbol $ b$, 
but there are error terms as well.  For the paraproducts, one needs a two-weight criterion of the boundedness. 
These criteria come in several different forms, but the additional structure of $ \mu , \lambda \in A_p$ 
forces these ostensibly different criteria to be jointly finite.  There are several error terms to handle. 
A comprehensive treatment of all terms depends upon an $ H ^{1}$-$BMO$ duality, which fortunately has 
already been developed.

Here is an outline of the paper.  Section \ref{s:NotationBackground} collects all the necessary background that will be used throughout the paper.   This includes background on $ A_p$ weights, and weighted $ H ^{1}$.  
Section \ref{s:Paraproducts} introduces the paraproduct operators of interest and proves they are bounded in terms of Bloom's $BMO$.  We are able to mimic certain unweighted proofs by using a duality statement for weighted $BMO$ spaces.  In Section \ref{s.EquivalenceBMO} we provide a family of equivalent conditions for a function to belong to the dyadic $BMO(\nu)$.  Some of these equivalences are more useful when obtaining lower bounds as in Theorem \ref{T:LowerBound}, while others are more important in the proof of the upper bound in Theorem \ref{T:UpperBound}. 
In particular, there is seemingly no canonical form of the definition of Bloom's $BMO$ space.  We have followed Bloom's presentation in the definition above, and find other forms of the definition more convenient at different points of the proof.  
We do not track $ A_p$ constants, since the sharp bound would depend upon the choice of norm for $ BMO (\nu )$; and we will freely use various equivalences throughout the proof below.   
Section \ref{s:UpperBound} contains the proof of Theorems \ref{T:UpperBound} and \ref{T:LowerBound} and Corollary \ref{C:Factorization}.  For the upper bound, we will use the Hyt\"onen Representation Theorem, \cite{HytRepOrig}, to decompose the Calder\'on--Zygmund operator into Haar shift operators.  Then we carefully analyze the commutator with each Haar shift to prove the desired statement in Theorem \ref{T:UpperBound}.  A similar proof strategy can be found in \cite{DalencOu}.  For the lower bound, we follow the original proof of Coifman, Rochberg, and Weiss, \cite{CRW}, but with suitable modifications.  The proof of Corollary \ref{C:Factorization} is then a standard application of well-known techniques.


\section{Notation and Background}
\label{s:NotationBackground}

Throughout this paper, we use the standard notation ``$A\lesssim B$'' to denote $A \leq cB$ for some constant $c$ that depends only on the dimension $n$ and, in the case of a weighted inequality, on $p$ and the $ A_p$ constants of $\mu$, and $\lb$.   And, ``$A\approx B$'' means that $A\lesssim B$ and $B\lesssim A$. We let ``$\defeq$'' mean equal by definition.


\subsection{Dyadic Grids} Recall the standard dyadic grid on $\R^n$:
	$$\mcd^0 \defeq \left\{ 2^{-k} \left([0,1)^n + m\right): k\in\mathbb{Z}; m \in\mathbb{Z}^n \right\}.$$
For every $\omega=(\omega_j)_{j\in\mathbb{Z}} \in (\{0,1\}^n)^{\mathbb{Z}}$ we may translate $\mcd^0$ by letting
	$$\mcd^{\omega} \defeq \left\{ Q \stackrel{\cdot}{+} \omega : Q\in\mcd^0 \right\},$$
where
	$$Q \stackrel{\cdot}{+} \omega \defeq Q + \sum_{j: 2^{-j}<l(Q)} 2^{-j}\omega_j.$$
Here $l(Q)$ will denote the side length of a cube $Q$ in $\R^n$. We will only need to pay attention to $\omega$ when dealing with $\mathbb{E}_{\omega}$, which denotes expectation with respect to the standard probability measure on the set of parameters $\omega$. We denote a generic dyadic grid $\mcd^{\omega}$ on $\R^n$ by $\mcd$.

Any such $\mcd$ has the standard nestedness properties:
	\begin{itemize}
	\item For every $P, Q \in \mcd$, $P \cap Q$ is one of $P$, $Q$, and $\emptyset$.
	\item All $Q\in\mcd$ with $l(Q) = 2^{-k}$ for some fixed $k\in\mathbb{Z}$ partition $\R^n$.
	\end{itemize}
For every $Q\in\mcd$ and every non-negative integer $k$, we denote:
	\begin{itemize}
	\item $Q^{(k)}$: the $k^{\text{th}}$ generation ancestor of $Q$ in $\mcd$, i.e. the unique element of $\mcd$ that contains $Q$ and has side length $2^k l(Q)$.
	\item $Q_{(k)}$: the collection of $k^{\text{th}}$ generation descendants of $Q$ in $\mcd$, i.e. the $2^{kn}$ disjoint subcubes of $Q$ in $\mcd$ with side length $2^{-k}l(Q)$.
	\end{itemize}


\subsection{The Haar System} Recall that every dyadic interval $I \subset \R$ is associated with two Haar functions:
	$$h_I^0 \defeq \frac{1}{\sqrt{|I|}} \left(\unit_{I_{-}} - \unit_{I_{+}}\right) \text{, and } h_I^1 \defeq \frac{1}{\sqrt{|I|}} \unit_I.$$
Note that $h_I^0$ is cancellative, while $h_I^1$ is non-cancellative. The cancellative Haar functions associated to a dyadic system on $\R$ form an orthonormal basis for $L^2(\R)$. 

More generally, let $Q = Q_1 \times \cdots \times Q_n$ be a dyadic cube in $\R^n$ -- here all $Q_i$ are dyadic intervals in $\R$ with common length $l(Q)$. Then $Q$ is associated with $2^n$ Haar functions:
	$$h_Q^{\epsilon}(x) \defeq h_{Q_1 \times \cdots \times Q_n}^{(\ep_1, \ldots, \ep_n)} (x_1, \ldots, x_n) 
		\defeq \prod_{i=1}^n h_{Q_i}^{\epsilon_i}(x_i),$$
where $\ep=(\epsilon_1,\ldots, \epsilon_n) \in \{0, 1\}^n$  is called the \textit{signature} of $h_Q^{\ep}$. We write $\ep\equiv 1$ when $\ep_i = 1$ for all $i$; in this case,
	$$h_Q^1 \defeq \frac{1}{\sqrt{|Q|}}\unit_Q$$
is non-cancellative. All the other $2^n - 1$ Haar functions $h_Q^{\ep}$ with $\ep \not\equiv 1$ associated with $Q$ are cancellative. Moreover, as in the one-dimensional case, all the cancellative Haar functions associated with a dyadic grid $\mcd$ on $\R^n$ form an orthonormal basis for $L^2(\R^n)$. In other words, every $f \in L^2(\R^n)$ has the expansion:
	$$f = \sum_{Q \in \mcd, \ep\not\equiv 1} \widehat{f}(Q, \ep) h_Q^{\ep},$$
where $ \widehat{f}(Q, \ep) \defeq \left< f, h_Q^{\ep}\right>$. 
Throughout this paper we use $\left<\cdot, \cdot\right>$ to denote the usual inner product on $L^2(\R^n)$. 

We make a few simple but useful observations about Haar functions. First, note that $h_Q^{\ep}$ is constant on any subcube $P \subsetneq Q$ of $Q$ in $\mcd$; we denote this value by $h_Q^{\ep}(P)$. Then for any integer $k \geq 1$, we can express $h_Q^{\ep}$ as:
	$$h_Q^{\ep} = \sum_{P \in Q_{(k)}} h_Q^{\ep}(P) \unit_P.$$
Second, a simple calculation shows that:
	$$h_Q^{\ep} h_Q^{\eta} = \frac{1}{\sqrt{|Q|}} h_Q^{\ep+\eta},$$
where for $\ep, \eta \in \{0, 1\}^n$ we define $\ep+\eta \in \{0, 1\}^n$ as:
	\begin{equation} \label{E:epadddef}
	(\ep+\eta)_i \defeq \delta_{(\ep_i, \eta_i)} 
		=\left\{ \begin{array}{ll} 
			0, & \text{if } \ep_i \neq \eta_i\\
			1, & \text{if } \ep_i = \eta_i.
			\end{array}\right.
	\end{equation}
It is easy to see from this definition that $\ep + \eta \equiv 1$ if and only if $\ep=\eta$, and $\ep + \eta = \ep$ if and only if $\eta\equiv 1$.
	
Third, we note that the average of a function $f$ over a dyadic cube $Q$:
	$$\left<f\right>_Q \defeq \frac{1}{|Q|} \int_Q f \,dx,$$
can be expressed as:
	$$\left<f\right>_Q = \sum_{\substack{P \in \mcd, P \supsetneq Q\\ \ep\not\equiv 1}} \widehat{f}(P, \ep) h_P^{\ep}(Q).$$
In turn, this yields the following useful expression:
	\begin{equation} \label{E:AvgDifference}
	\left<f\right>_Q - \left<f\right>_{Q^{(i)}} = \sum_{\substack{P,Q \in \mcd, \ep\not\equiv 1\\ Q \subsetneq P \subset Q^{(i)}}} \widehat{f}(P, \ep) h_P^{\ep}(Q) = \sum_{\substack{1\leq k\leq i\\ \ep\not\equiv 1}} \widehat{f}(Q^{(k)}, \ep) h_{Q^{(k)}}^{\ep}(Q),
	\end{equation}
which we shall use later in the proof of our main result.
	

\subsection{$A_p$ Weights} 
\label{ss:apweights}
Let $w$ be a weight on $\R^n$, i.e. $w$ is an almost everywhere positive, locally integrable function. For $1 < p < \infty$, let $L^p(w) \defeq L^p(\R^n; w)$ be the space of functions $f$ that satisfy:
	$$\|f\|_{L^p(w)} \defeq \left(\int_{\R^n} |f(x)|^p\,dw(x)\right)^{\frac{1}{p}} < \infty,$$
where we also use $w$ to denote the measure $w(x)\,dx$. For a cube $Q$ in $\R^n$, we let 
	$$w(Q) \defeq \int_Q w(x)\,dx \text{ and } \left<w\right>_Q \defeq \frac{w(Q)}{|Q|}.$$
We say that $w$ belongs to the Muckenhoupt class of $A_p$ weights for some $1 < p < \infty$ provided that:
	$$[w]_{A_p} \defeq \sup_{Q} \left<w\right>_Q \left<w^{1-q}\right>_Q^{p-1} < \infty,$$
where $q$ denotes the H\"older conjugate of $p$ and the supremum above is over all cubes $Q$ in $\R^n$ with sides parallel to the axes. The quantity $[w]_{A_p}$ is called the $A_p$ (Muckenhoupt) characteristic of $w$. 

If $w \in A_p$, then the `conjugate' weight 
\begin{equation}\label{e:conjugate}
w' \defeq w^{1-q} \in A_q,
\end{equation} 
with $A_q$ characteristic $[w']_{A_q} = [w]_{A_p}^{q-1}$. In other words:
	\begin{equation} \label{E:ApWeightIneqs}
	1 \leq \left<w\right>_Q\left<w'\right>_Q^{p-1} \leq [w]_{A_p} \:\:\text{ and }\:\: 1 \leq \left<w'\right>_Q\left<w\right>_Q^{q-1} \leq [w]_{A_p}^{q-1},
	\end{equation}
for all $w\in A_p$ and all cubes $Q$ in $\R^n$. We shall make much use of the duality relationship:
	\begin{equation} \label{E:ApDualSpace}
	\left(L^p(w)\right)^* \equiv L^q(w') \text{, with pairing } \left< f, g\right> \text{, for all } f \in L^p(w), g\in L^q(w'),
	\end{equation}
and 
	$$\|f\|_{L^p(w)} = \sup_{\substack{g\in L^q(w')\\ \|g\|_{L^q(w')} \leq 1}} |\La f, g\Ra|.$$
The case $p=2$ is particularly easy to work with, as $w' = w^{-1}$ for $w\in A_2$.

A crucial property of $A_p$ weights that we shall use repeatedly is the $L^p(w)$-boundedness of the maximal function:
	$$Mf \defeq \sup_{Q \text{ cubes in } \R^n} \left(\left<|f|\right>_Q \unit_Q\right).$$
Muckenhoupt  \cite{MR0293384} showed that 
	\begin{equation}\label{E:MaxFWeight}
	\|Mf\|_{L^p(w)} \lesssim  \|f\|_{L^p(w)}.
	\end{equation}
The sharp behavior in terms of the $A_p$ characteristic of the maximal function on $L^p(w)$ was obtained by Buckley in \cite{MR1124164}.

Another pivotal development in $A_p$ weight theory was the Extrapolation Theorem - see \cite{Extrapolation} - which, in particular, allows one to deduce the $L^p(w)$-boundedness of an operator for all $w\in A_p$ solely from its $L^2(w)$-boundedness for all $w\in A_2$. This is an extremely useful tool because, as we shall see, $L^2$-estimates for $A_2$ weights are usually much `easier' than $L^p$-estimates for $A_p$ weights. 

As mentioned in the Introduction, we are not explicitly tracking the dependence upon the weights $\mu$ and $\lambda$ in terms of the $A_p$ characteristic since we will freely use various equivalences between various norms in the proof below.  However, certain constants need to be tracked as they play a role in the final analysis carried out later (see for instance Lemma \ref{L:ShiftedDSF} and estimate \eqref{E:SijUBd}).  The main tool we shall use is the following form of the Extrapolation Theorem:
	
	\begin{thm}\label{T:Extrapolation}
	Suppose an operator $T$ satisfies:
		$$\|Tf\|_{L^2(w)} \leq A  C (w)  \|f\|_{L^2(w)}$$
	for all $w\in A_2$, for some fixed $A> 0$. Then:
		$$\|Tf\|_{L^p(w)} \le A  C (w, p)\|f\|_{A_p}$$
	for all $1<p<\infty$ and all $w\in A_p$.
	\end{thm}


\subsection{Dyadic Square Functions} Given a dyadic grid $\mcd$ on $\R^n$, the dyadic square function $S_{\mcd}$ is defined by:
	$$S_{\mcd}f \defeq \left[\, \sum_{Q\in\mcd,\ep\neq 1} |\widehat{f}(Q,\ep)|^2 \frac{\unit_Q}{|Q|} \right]^{\frac{1}{2}}.$$
A crucial property of this operator is the equivalence of norms
\begin{equation}\label{e:SQ}
\|f\|_{L^p(w)} \simeq \|S_{\mcd}f\|_{L^p(w)},
\end{equation}
for $w\in A_p$, $1<p<\infty$.  Sharp behavior of the square function in terms of the $A_p$ characteristic can be found in \cite{MR2200743}.

We will also need the following weighted estimate for a shifted square function:


\begin{lm} \label{L:ShiftedDSF}
For a dyadic grid $\mcd$ on $\R^n$ and a pair $(i, j)$ of non-negative integers, define:
	\begin{equation} \label{E:ShiftedDSFDef}
	\widetilde{S_{\mcd}}^{i,j} f \defeq \left (\sum_{Q\in\mcd, \ep\not\equiv 1} \left(\sum_{P \in (Q^{(j)})_{(i)}} |\widehat{f}(P,\ep)| \right)^2 
		\frac{\unit_Q}{|Q|}\right)^{\frac{1}{2}}.
	\end{equation}
Then for every weight $w\in A_p$, with $1 < p < \infty$:
	\begin{equation} \label{E:ShiftedDSFUB}
	\left\| \widetilde{S_{\mcd}}^{i, j} : L^p(w) \rightarrow L^p(w)\right\|
		\lesssim 2^{\frac{n}{2}(i+j)}.
	\end{equation}
\end{lm}

\noindent Remark that for $i=j=0$, this is just the usual dyadic square function $S_{\mcd}$.

\begin{proof}
In light of the Extrapolation Theorem \ref{T:Extrapolation}, it suffices to prove an upper bound for all $A_2$ weights $w$. So let $w \in A_2$ and note that
	\begin{align*}
	\|\widetilde{S_{\mcd}}^{i,j}f\|^2_{L^2(w)} & = \sum_{Q\in\mcd, \ep\not\equiv 1} 
			\left(\sum_{P \in (Q^{(j)})_{(i)}} |\widehat{f}(P,\ep)|\right)^2 \La w\Ra_Q  \\
		& = \sum_{R \in \mcd, \ep\not\equiv 1} \left( \sum_{P \in R_{(i)}} |\widehat{f}(P,\ep)|\right)^2 
			\sum_{Q\in R_{(j)}} \La w\Ra_Q.
	\end{align*}
Now
	\begin{align*}
	\sum_{P \in R_{(i)}} |\widehat{f}(P,\ep)| & = \sum_{P \in R_{(i)}} |\widehat{f}(P,\ep)| \frac{\La w^{-1}\Ra_P^{\frac{1}{2}}}{\La w^{-1}\Ra_P^{\frac{1}{2}}} 
			\\
	& \leq \left( \sum_{P \in R_{(i)}} |\widehat{f}(P,\ep)|^2 \frac{1}{\La w^{-1}\Ra_P}
	 \times \sum_{P\in R_{(i)}} \La w^{-1}\Ra_P \right)^{\frac{1}{2}},
	\end{align*}
so, appealing to the square function bound \eqref{e:SQ}, 
	\begin{align*}
	\|\widetilde{S_{\mcd}}^{i,j}f\|^2_{L^2(w)} & \leq 2^{n(i+j)} \sum_{R\in\mcd,\ep\not\equiv 1} 
		 \sum_{P \in R_{(i)}} |\widehat{f}(P,\ep)|^2 \frac{1}{\La w^{-1}\Ra_P}
		\La w^{-1}\Ra_R \La w\Ra_R & \\
	& \lesssim 2^{n(i+j)}  \sum_{P\in\mcd,\ep\not\equiv 1} |\widehat{f}(P,\ep)|^2 \frac{1}{\La w^{-1}\Ra_P} 
	\lesssim 2^{n(i+j)}   \|f\|_{L^2(w)}^2.
	\end{align*}
\end{proof}


\subsection{Hyt\"onen's Representation Theorem} Fix a dyadic grid $\mcd^{\omega}$ on $\R^n$. For every pair $i,j$ of non-negative integers, a dyadic shift operator with parameters $(i,j)$ is an operator of the form:
	$$\mbs_{\omega}^{ij}f \defeq \sum_{\substack{R\in\mcd\\ \ep,\eta\in\{0,1\}^n}} \sum_{\substack{P \in R_{(i)}\\Q \in R_{(j)}}} a^{\ep\eta}_{PQR}
		\widehat{f}(P, \ep) h_Q^{\eta},$$
where $a^{\ep\eta}_{PQR}$ are coefficients with
	$$|a^{\ep\eta}_{PQR}| \leq \frac{\sqrt{|P||Q|}}{|R|} = 2^{-\frac{n}{2}(i+j)}.$$
The operator $\mbs_{\omega}^{ij}$ is called cancellative if all Haar functions appearing in its definition are cancellative. Otherwise, $\mbs_{\omega}^{ij}$ is called non-cancellative. 
The parameters $ \kappa =(i,j)$ are a measure of the \emph{complexity} of the shift.  As is well-known, 
the dependence of norm estimates upon complexity must be tracked, but is only linear in 
\begin{equation}\label{e:kappa}
\kappa_{ij} \defeq \max(i, j, 1), 
\end{equation}
whereas there is exponential decay in $ \kappa $, in the 
celebrated representation theorem of Hyt{\"o}nen \cites{HytRepOrig, HytRep, HytPerezTV}:

\begin{thm} \label{T:HytRep}
Let $T$ be a Calder{\'o}n-Zygmund operator associated with a $\delta$-standard kernel. Then there exist dyadic shift operators $\mbs_{\omega}^{ij}$ with parameters $(i, j)$ for all non-negative integers $i, j$ such that
	$$\La Tf, g\Ra = c\: \mathbb{E}_{\omega} \sum_{i,j=0}^{\infty} 2^{- \kappa _{i,j}\frac{\delta}{2}} \La \mbs_{\omega}^{ij}f, g\Ra,$$
for all bounded, compactly supported functions $f$ and $g$, where $c$ is a constant depending on the dimension $n$ and on $T$. Here all $\mbs^{ij}_{\omega}$ with $(i,j) \neq (0,0)$ are cancellative, but the shifts $\mbs_{\omega}^{00}$ may be non-cancellative.
\end{thm}

The statement of this Theorem involves a \emph{random} choice of grids.  However, in all applications of this result, 
one analyzes the norm behavior of the Haar shift operators, establishing bounds that are uniform with respect to the choice of dyadic grid.  The exact manner in which the random dyadic grid are formed is not relevant to us.  
We will discuss the case $i=j=0$ in more detail in Section \ref{Ss:00Remainder}. 

Another useful tool for us will be the weighted estimate below, which can be found in \cites{HytLacey,Lacey,TreilSharpA2}.

\begin{thm} \label{T:SijWeighted}
Let $\mbs_{\omega}^{ij}$ be a dyadic shift operator with complexity $\kappa_{ij}$. Then for any weight $w\in A_p$ with $p>1$:
	\begin{equation} \label{E:SijWeightedUB}
	\left\|\mbs_{\omega}^{ij} : L^p(w) \rightarrow L^p(w)\right\| \lesssim \kappa_{ij} [w]_{A_p}^{\max\left(1, \frac{1}{p-1}\right)}. 
	\end{equation}
\end{thm}

\subsection{Weighted $BMO$-$H^1$ duality} \label{Ss:WeightedBMO}

For a weight $w$ on $\R^n$, the weighted $BMO$ space $BMO(w)$ is defined to be the space of all locally integrable functions $b$ that satisfy:
	\begin{equation} \label{E:wBMONorm1}
	\|b\|_{BMO(w)} \defeq \sup_{Q} \frac{1}{w(Q)} \int_Q |b - \left<b\right>_Q|\,dx  < \infty,
	\end{equation}
where the supremum is over all cubes $Q$ in $\R^n$ with sides parallel to the axes. 
For a general weight, the definition of the $BMO$ norm is highly dependent on its $ L ^{1}$ average.  
But, if the weight is $ A_ \infty $, one is free to replace the $ L ^{1}$-norm by larger averages, 
though this must be done with a little care.
Define  for $ 1\leq q < \infty $, 
	\begin{equation} \label{E:wBMONormq}
	\|b\|_{BMO^q(w)} ^{q} \defeq \sup_{Q}   \frac{1}{w(Q)} \int_Q |b - \left<b\right>_Q|^q\,dw' .
	\end{equation}
Above, note that the conjugate weight is used in place of Lebesgue measure in \eqref{E:wBMONorm1}.  

\begin{lm}\label{l:MW}\cite{MuckWheeden}*{Thm. 4}  
With the notation above, there holds 
\begin{equation} \label{E:wBMO-qEquiv}
	\|b\|_{BMO(w)} \leq \|b\|_{BMO^q(w)} \lesssim  \|b\|_{BMO(w)}, \qquad   1\leq q \leq 2. 
	\end{equation}
On the right, the implied constant depends upon $ q$ and $ [w] _{A _{2 }}$.  
(And the right inequality is false in general for $ q> 2$.) 
\end{lm}
The first inequality follows from  H\"older's inequality and $ 1\leq q \leq 2$,  while  the second is more involved.  
The essential point for us is the case $ q=2$ above.

For a dyadic grid $\mcd$ on $\R^n$, we define the dyadic versions of the norms above by taking supremum over $Q \in \mcd$ instead of over all cubes $Q$ in $\R^n$, and denote these spaces by $\wbmod$ and $BMO^q_{\mcd}(w)$. Clearly $BMO(w) \subset \wbmod$ for any choice of $\mcd$, and the equivalence in \eqref{E:wBMO-qEquiv} also holds for the dyadic versions of these spaces. 
	
Now fix a dyadic grid $\mcd$ on $\R^n$ and a weight $ w\in A_ \infty $. Define the dyadic weighted Hardy space $H_{\mcd}^1(w)$ (see \cite{GarciaCuerva}) to be the space of all $\Phi$ that satisfy:
	$$\|\Phi\|_{H_{\mcd}^1 (w)} \defeq \|S_{\mcd}\Phi\|_{L^1(w)} < \infty.$$
The dual space of $H_{\mcd}^1(w)$ is the weighted Carleson measure space $CM_{\mcd}^1(w)$, that is, the space of all locally integrable functions $g$ such that:
	$$\|g\|_{CM_{\mcd}^1(w)} \defeq \sup_{Q\in\mcd} \left( \frac{1}{w(Q)} \sum_{\substack{P\subset Q\in\mcd\\ \ep\neq 1}} \frac{|\widehat{g}(P, \ep)|^2 }{\left<w\right>_P} \right)^{\frac{1}{2}} < \infty,$$
with duality pairing $\left<g, \Phi\right>$ for $g\in CM_{\mcd}^1(w)$ and $\Phi\in H^1_{\mcd}(w)$ see \cites{LeeLinLin,Wu}. We have then
	\begin{equation*} 
	\left|\left< g, \Phi \right>\right| \leq \|g\|_{CM_{\mcd}^1(w)} \|S_{\mcd}\Phi\|_{L^1(w)}.
	\end{equation*}

Specializing to the case of $ w\in A_2$, we have an $ H ^{1}$-$BMO$ duality.  
	
\begin{lm}\label{l:wA2} If $ w \in A_2$, there holds 
	\begin{equation} \label{E:H1BMODual1}
	\left|\La b, \Phi\Ra\right| \lesssim   \|b\|_{\wbmosd} \|S_{\mcd}\Phi\|_{L^1(w)}.
	\end{equation}
\end{lm}

\begin{proof}
The inequality follows from $ \|b\|_{CM^1_{\mcd}(w)} \lesssim   \|b\|_{\wbmosd}$. 
 But observe that fixing a cube $ Q$, and expanding $ b$ in the Haar basis, we have 
	\begin{equation} \label{E:BMOHaarExp}
	(b - \La b\Ra_Q)\unit_Q = \sum_{P \subset Q \in\mcd, \ep\neq 1} \widehat{b}(P,\ep) h_P^{\ep}=: B_Q. 
	\end{equation}
Then, by \eqref{e:SQ},  
	\begin{align*}
	\int_Q |b - \avgb_Q|^2 \,dw^{-1} & = \|B_Q\|^2_{L^2(w^{-1})} \\
	& \gtrsim   \|S_{\mcd}B_Q\|^2_{L^2(w^{-1})} \\
	& =   \sum_{P\subset Q\in\mcd,\ep\neq 1} |\widehat{b}(P,\ep)|^2 \La w^{-1}\Ra_P & \\
	& \gtrsim  \sum_{P\subset Q\in\mcd,\ep\neq 1} |\widehat{b}(P,\ep)|^2 \frac{1}{\La w\Ra_P}. &
	\end{align*}
And so the Lemma follows. 
\end{proof}


\subsection{Bloom's $BMO(\nu)$} From here on, fix $1<p<\infty$ and two $A_p$ weights $\mu$ and $\lb$ on $\R^n$, 
and define Bloom's weight 
	$$\nu \defeq \mu^{\frac{1}{p}} \lb^{-\frac{1}{p}}.$$

\begin{lm}\label{l:} The weight $ \nu$ belongs to the $A_2$ class. In particular: 
$$[\nu]_{A_2} \leq [\mu]_{A_p}^{\frac{1}{p}} [\lb]_{A_p}^{\frac{1}{p}}.$$
Moreover
\begin{equation}
\label{E:NuH1BMODuality}
	\left|\La b, \Phi\Ra\right|  \lesssim   \|b\|_{BMO^2_{\mcd}(\nu)} \|S_{\mcd}\Phi\|_{L^1(\nu)}.
\end{equation}
\end{lm}
	
\begin{proof}
The second inequality is immediate from \eqref{E:H1BMODual1}.  For the first, 
by H\"older's inequality:
	\begin{equation} \label{E:NuIneqs}
	\La \nu\Ra_Q \leq \La\mu\Ra_Q^{\frac{1}{p}}\La\lb'\Ra_Q^{\frac{1}{q}} \text{\:\: and \:\:} \La\nu^{-1}\Ra_Q \leq \La\mu'\Ra_Q^{\frac{1}{q}}\La\lb\Ra_Q^{\frac{1}{p}},
	\end{equation}
so
	\begin{align*}
	\La \nu\Ra_Q \La\nu^{-1}\Ra_Q & \leq \left(\La\mu\Ra_Q\La\mu'\Ra_Q^{p-1}\right)^{\frac{1}{p}} \left(\La\lb\Ra_Q\La\lb'\Ra_Q^{p-1}\right)^{\frac{1}{p}}\\
	\end{align*}
But the terms in parentheses are at most  $ [ \mu ] _{A_p}$ and $ [ \lambda ] _{A_p}$.  Hence $ \nu \in A_2$.  
	\end{proof}

We record here a simple inequality. 
	\begin{equation} \label{E:NuIneqs1}
	\La\mu\Ra_Q^{\frac{1}{p}}\La\lb'\Ra_Q^{\frac{1}{q}} \lesssim 
		\frac{1}{\La\mu'\Ra_Q^{\frac{1}{q}}\La\lb\Ra_Q^{\frac{1}{p}}} \lesssim 
		\frac1{\La \nu^{-1}\Ra_Q} \lesssim  \La \nu\Ra_Q.
	\end{equation}


\section{Two-Weight Inequalities for Paraproduct Operators}
\label{s:Paraproducts}


\subsection{Paraproducts} For a fixed dyadic grid $\mcd$ on $\R^n$, the `paraproduct' operators with symbol $b$ are defined by:
\begin{align}\label{e:para}
\Pi_b^{\mcd} f &\defeq \sum_{Q \in \mcd, \ep \neq 1} \widehat{b}(Q, \ep) \left<f\right>_Q h_Q^{\ep},
\\
\Pi_b^{*\,{\mcd}} f &\defeq \sum_{Q \in \mcd, \ep \neq 1} \widehat{b}(Q, \ep) \widehat{f}(Q, \ep) \frac{\unit_Q}{|Q|},
\\ \label{e:gamma} 
\textup{and} \qquad 
\Gamma^{\mcd}_b f &\defeq \sum_{Q \in \mcd} \sum_{\substack{\ep, \eta \neq 1\\ \ep \neq \eta}} \widehat{b}(Q, \ep) \widehat{f}(Q, \eta) \frac{1}{\sqrt{|Q|}} h_Q^{\ep + \eta}.
\end{align}
For ease of notation, we fix $\mcd$ through the rest of this section and suppress the subscript $\mcd$ from the paraproducts.
	
Commutators are a difference of products, and  the product of two functions is decomposed into  paraproducts as follows:
	\begin{equation} \label{E:ParaprodDecomp}
	bf = \Pi_b f + \Pi_f b + \Pi_b^* f + \Gamma_b f.
	\end{equation}

To see the decomposition in \eqref{E:ParaprodDecomp}, express $b$ and $f$ in terms of the Haar expansions:
	$$bf = 
	\sum_{R, Q\in\mcd}\sum_{\ep,\eta\neq 1} \widehat{b}(Q, \ep)\widehat{f}(R, \eta)h_Q^{\ep}h_R^{\eta}$$
and analyze the sum over the three different cases $R \subsetneq Q$, $R \supsetneq Q$, and $Q = R$. The latter case easily yields:
	$$
		\sum_{Q\in\mcd,\ep\neq 1} \widehat{b}(Q, \ep) \widehat{f}(Q, \ep) \frac{\unit_Q}{|Q|} +
		\sum_{Q\in\mcd,\ep\not\equiv 1}\sum_{\eta\not\equiv 1, \ep\neq\eta} \widehat{b}(Q, \ep)\widehat{f}(Q, \eta)\frac{1}{\sqrt{|Q|}}h_Q^{\ep+\eta} 
		= \Pi_b^* f + \Gamma_b f.$$
To illustrate one of the other two cases:
	\begin{eqnarray*}
	\sum_{Q \subsetneq R\in\mcd;\ep,\eta\neq 1} \widehat{b}(Q, \ep)\widehat{f}(R, \eta) h_Q^{\ep}h_R^{\eta} &=& 
		\sum_{Q \subsetneq R\in\mcd;\ep,\eta\neq 1} \widehat{b}(Q, \ep)\widehat{f}(R, \eta) h_Q^{\ep}h_R^{\eta}(Q) \\
	&=& \sum_{Q\in\mcd,\ep\neq 1} \widehat{b}(Q, \ep)h_Q^{\ep} \sum_{R\supsetneq Q, \eta\neq 1} \widehat{f}(R, \eta)h_R^{\eta}(Q)\\
	&=& \sum_{Q\in\mcd,\ep\neq 1} \widehat{b}(Q, \ep) \left<f\right>_Q h_Q^{\ep}= \Pi_b f.
	\end{eqnarray*}
Similarly, the case $R\subsetneq Q$ yields $\Pi_f b$.

We pause here for a moment to remark that the term $\Gamma_b$ disappears in the one-dimensional case, where the familiar decomposition is $bf = \Pi_bf + \Pi_fb + \Pi^*_bf$. By definition \eqref{E:epadddef}, $\ep + \eta = 1$ if and only if $\ep = \eta$, so while $\Pi^*_bf$ maintains its one-dimensional structure and contains all the non-cancellative Haar functions, the term $\Gamma_b$ contains only the cancellative Haar functions. Moreover, as in the one-dimensional case, $\Pi_b^*$ is the adjoint of $\Pi_b$ in \textit{unweighted} $L^2(\R^n)$, while the third paraproduct $\Gamma_b$ is self-adjoint in $L^2(\R^n)$:
	\begin{equation} \label{E:UnweightedAdjoints}
	\La \Pi_b f, g\Ra = \La f, \Pi^*_b g\Ra \:\:\:\:\text{and}\:\:\:\: \La \Gamma_b f, g\Ra = \La f, \Gamma_b g\Ra.
	\end{equation}

\subsection{Two-weight Inequalities for Paraproducts} Next, we discuss boundedness of the paraproducts as operators from $L^p(\mu) \rightarrow L^p(\lb)$. Before we proceed, we make the interesting observation that the adjointness statements about the three paraproducts in unweighted $L^2(dx)$ extend to this case, in the sense of Banach space adjoints. Specifically
\begin{align*}
& \text{The adjoint of } \Pi_b : L^p(\mu) \rightarrow L^p(\lb) \text{ is } \Pi^*_b : L^q(\lb') \rightarrow L^q(\mu');\\
& \text{The adjoint of } \Pi^*_b : L^p(\mu) \rightarrow L^p(\lb) \text{ is } \Pi_b : L^q(\lb') \rightarrow L^q(\mu');\\
& \text{The adjoint of } \Gamma_b : L^p(\mu) \rightarrow L^p(\lb) \text{ is } \Gamma_b : L^q(\lb') \rightarrow L^q(\mu').
\end{align*}
Here, $ \lambda '$ is the conjugate weight, as in \eqref{e:conjugate}.

These follow from \eqref{E:ApDualSpace}. For instance, the adjoint of $\Pi_b : L^p(\mu) \rightarrow L^p(\lb)$ is the unique operator $T : L^q(\lb') \rightarrow L^q(\mu')$ such that 
	$$\La \Pi_b f, g\Ra = \La f, Tg\Ra \text{, for all } f \in L^p(\mu), g\in L^q(\lb').$$
But this is just the inner product in unweighted $L^2(dx)$, so $\La \Pi_b f, g\Ra = \La f, \Pi^*_b g\Ra$, and $T = \Pi^*_b$. The second statement above follows identically, and the third statement follows from the self-adjointness of $\Gamma_b$ in $L^2(dx)$. 

This is a two-weight result for paraproducts, fundamental for us.  

\begin{thm} \label{T:ParaprodTwoWeight}
Let $\mcd$ be a fixed dyadic grid on $\R^n$, and suppose $b \in \nbmosd$. Then:
	\begin{align} \label{E:PibUBd}
	\left\|\Pi_b : L^p(\mu) \rightarrow L^p(\lb)\right\| = \left\|\Pi^*_b : L^q(\lb') \rightarrow L^q(\mu')\right\| &\lesssim
\|b\|_{BMO^2_{\mcd}(\nu)},
\\
 \label{E:PibStarUBd}
	\left\|\Pi^*_b : L^p(\mu) \rightarrow L^p(\lb)\right\| = \left\|\Pi_b : L^q(\lb') \rightarrow L^q(\mu')\right\| & \lesssim
\|b\|_{BMO^2_{\mcd}(\nu)},
\\
\label{E:GammabUBd}
	\left\|\Gamma_b : L^p(\mu) \rightarrow L^p(\lb)\right\| = \left\|\Gamma_b : L^q(\lb') \rightarrow L^q(\mu')\right\| 
& \lesssim  \|b\|_{BMO^2_{\mcd}(\nu)}. 
	\end{align}
\end{thm}

\begin{proof} 
The proof is by duality, exploiting the $ H^1$-$BMO$ duality inequality \eqref{E:NuH1BMODuality} 
to gain the term $ \|b\|_{BMO^2_{\mcd}(\nu)}$. This will leave us with a bilinear square function involving $ f$ and $ g$, 
which will be controlled by a product of a maximal function and a linear square function. The details are as follows. 
We let $f \in L^p(\mu)$ and $g \in L^q(\lb')$. Then
	\begin{align*}
	|\La \Pi_b f, g\Ra| & = \left| \sum_{Q\in\mcd,\ep\neq 1} \widehat{b}(Q,\ep) \La f\Ra_Q \widehat{g}(Q, \ep) \right| & \\
		& = |\La b, \Phi\Ra| & \text{ where } \Phi \defeq \sum_{Q\in\mcd,\ep\neq 1} \La f\Ra_Q \widehat{g}(Q,\ep) h_Q^{\ep}\\
		& \lesssim  \|b\|_{BMO^2_{\mcd}(\nu)} \|S_{\mcd} \Phi\|_{L^1(\nu)} & \text{ by \eqref{E:NuH1BMODuality}.}
	\end{align*}
Now, $ S_{\mcd}\Phi$ is bilinear in $ f$ and $ g$, and is no more than 
	$$(S_{\mcd}\Phi)^2  = \sum_{Q\in\mcd,\ep\neq 1} |\La f\Ra_Q|^2 |\widehat{g}(Q,\ep)|^2 \frac{\unit_Q}{|Q|}
		 \leq (Mf)^2 \sum_{Q\in\mcd,\ep\neq 1} |\widehat{g}(Q,\ep)|^2 \frac{\unit_Q}{|Q|}
		 = (Mf)^2 (S_{\mcd}g)^2.$$
A straight forward application of H\"older's inequality, and bounds for the maximal and square functions will complete the proof. 
	\begin{align*}
	\|S_{\mcd}\Phi\|_{L^1(\nu)} & \leq \int (Mf) (S_{\mcd}g) \,d\mu^{\frac{1}{p}} \lb^{-\frac{1}{p}} & \\
		& \leq \|Mf\|_{L^p(\mu)} \|S_{\mcd}g\|_{L^q(\lb')} \lesssim \|f\|_{L^p(\mu)}  \|g\|_{L^q(\lb')}
	\end{align*}
	 by \eqref{E:MaxFWeight}, \eqref{e:SQ}. 
	 This completes the proof of \eqref{E:PibUBd}.

\medskip 

The second set of inequalities  \eqref{E:PibStarUBd} are equivalent to the first, by a simple duality argument.  
Concerning the last set of inequalities,  \eqref{E:GammabUBd}, they are different in that the operator 
 only has cancellative Haar functions. One can bound Haar coefficients by 
maximal functions, doing so on either $ f$ or $ g$.  

\end{proof}


\section{Equivalences for Dyadic Bloom $BMO$}
\label{s.EquivalenceBMO}
Bloom's $ BMO$ space has several equivalent formulations, which is a key fact in proof of the main theorems.  
Those that we need are summarized here.  
For a fixed dyadic grid $\mcd$ on $\R^n$ define the quantities:
	\begin{equation} \label{E:B1D}
	\mathbb{B}^{\mcd}_{1}(b, \mu, \lb) \defeq \sup_{Q\in\mcd} \left(\frac{1}{\mu(Q)} \int_Q |b - \La b\Ra_Q|^p \,d\lb\right)^{\frac{1}{p}};
	\end{equation}
	\begin{equation} \label{E:B2D}
	\mathbb{B}^{\mcd}_{2}(b, \mu', \lb') \defeq \sup_{Q\in\mcd} \left(\frac{1}{\lb'(Q)} \int_Q |b - \La b\Ra_Q|^q \,d\mu'\right)^{\frac{1}{q}}.
	\end{equation}
We provide several equivalent statements for the dyadic version of Bloom's $BMO$ space $BMO_{\mcd}(\nu)$. 

\begin{thm} \label{T:BloomEquiv}
Let $\mcd$ be a fixed dyadic grid on $\R^n$. The following are equivalent:
	\begin{enumerate}[{\normalfont (1)}]
	\item $b \in BMO_{\mcd}^2(\nu)$.
	\item The operator $\Pi_b : L^p(\mu) \rightarrow L^p(\lb)$ is bounded.
	\item The operator $\Pi^*_b : L^p(\mu) \rightarrow L^p(\lb)$ is bounded.
	\item The operators $\Pi_b$ and $\Pi^*_b$ are bounded $L^2(\nu) \rightarrow L^2(\nu^{-1})$.
	\item $\mathbb{B}^{\mcd}_1(b, \mu, \lb) < \infty$.
	\item $\mathbb{B}^{\mcd}_2(b, \mu', \lb') < \infty$.
	\item $b \in BMO_{\mcd}(\nu)$.
	\end{enumerate}
\end{thm}

\begin{proof}
\noindent $(1) \Rightarrow (2)$ and $ (3)$. 
This is the core of Theorem \ref{T:ParaprodTwoWeight}.

\vspace{0.1in}

\noindent $(2) \Rightarrow (5)$ and $(3)\Rightarrow (6)$: 
These two assertions are the same by duality, and we consider the first implication.  
Assuming the paraproduct $ \Pi_b $ is bounded, we test the norm of this operator on indicators of intervals, 
and get the condition in (5).  By the Littlewood-Paley inequalities in weighted $ L ^{p}$ spaces,  \eqref{e:SQ}, we have 
\begin{align*}
 \|\unit_Q(b - \La b\Ra_Q)\|_{L^p(\lb)} 
 & \lesssim  \|S_{\mcd} [\unit_Q(b-\La b\Ra_Q)]\|_{L^p(\lb)} 
 \\ 
 & \leq  \|S_{\mcd}(\Pi_b\unit_Q)\|_{L^p(\lb)} 
\end{align*}
Above, we have the square function of $ \Pi_b\unit_Q$, which follows from the identity 
\begin{equation*}
\Pi_b\unit_{Q} = \sum_{P\subset Q, \ep\not\equiv 1} \widehat{b}(P,\ep) h_P^{\ep} + 
		\sum_{P\supsetneq Q, \ep\not\equiv 1} \widehat{b}(P,\ep)\frac{|Q|}{|P|}h_P^{\ep}. 
\end{equation*}
Again by the Littlewood-Paley inequalities, 
\begin{equation*}
 \|S_{\mcd}(\Pi_b\unit_Q)\|_{L^p(\lb)}  \lesssim  \|\Pi_b \unit_Q\|_{L^p(\lb)}.
\end{equation*}
But, the assumption of the norm boundedness of the paraproduct implies that we have 
\begin{equation*}
\lambda (Q) ^{-1/p}  \|S_{\mcd} [\unit_Q(b-\La b\Ra_Q)]\|_{L^p(\lb)}  
\lesssim   \lVert \Pi_b : L^p(\mu) \rightarrow L^p(\lb)\rVert. 
\end{equation*}
	
\vspace{0.1in}
\noindent $(5) $ or $(6) \Rightarrow (7)$: Suppose $\mathbb{B}^{\mcd}_1(b, \mu, \lb) < \infty$, the case of 
$\mathbb{B}^{\mcd}_2(b, \mu', \lb') < \infty$ being similar. Then
	\begin{align*}
	\int_Q |b - \La b\Ra_Q|\,dx & \leq \left(\int_Q |b - \La b\Ra_Q|^p\,d\lb\right)^{\frac{1}{p}} 
		\left(\int_Q \,d\lb'\right)^{\frac{1}{q}} & \\
	& \lesssim  \mathbb{B}^{\mcd}_1(b, \mu,\lb) \mu(Q)^{\frac{1}{p}} \lb'(Q)^{\frac{1}{q}} & \\
	& \lesssim  \mathbb{B}^{\mcd}_1(b, \mu,\lb) \nu(Q) & \text{by \eqref{E:NuIneqs1}}.
	\end{align*}
That is, we have shown 
$\|b\|_{BMO_{\mcd}(\nu)} \leq \mathbb{B}^{\mcd}_1(b, \mu,\lb) $. 
	
\vspace{0.1in}
\noindent $(7) \Leftrightarrow (1)$: As discussed in Section \ref{Ss:WeightedBMO}, this is proved in \cite{MuckWheeden}.

\vspace{0.1in}
\noindent $(1) \Rightarrow (4)$: Suppose $b \in BMO_{\mcd}^2(\nu)$. We make use of the duality $(L^2(\nu))^* \equiv L^2(\nu^{-1})$, with the usual unweighted $L^2$ inner product as the duality pairing, and let $f, g \in L^2(\nu)$. Then
	$$|\La \Pi_b f, g\Ra| = |\La b, \Phi\Ra| \lesssim  \|b\|_{BMO^2_{\mcd}(\nu)} \|S_{\mcd}\Phi\|_{L^1(\nu)},$$
where $\Phi = \sum_{P,\ep} \La f\Ra_P \widehat{g}(P,\ep) h_P^{\ep}$. It follows easily that $S_{\mcd}\Phi \leq (Mf)(S_{\mcd}g)$, and then from \eqref{E:MaxFWeight} and \eqref{e:SQ}:
	$$\|S_{\mcd}\Phi\|_{L^1(\nu)} \leq \|Mf\|_{L^2(\nu)} \|S_{\mcd}f\|_{L^2(\nu)} \lesssim 
\|f\|_{L^2(\nu)}  \|g\|_{L^2(\nu)}.$$
Then
	$$\|\Pi_b : L^2(\nu) \rightarrow L^2(\nu^{-1}) \| \lesssim  \|b\|_{BMO^2_{\mcd}(\nu)}.$$
The same statement for $\Pi^*_b$ follows by noting that $\Pi_b^* : L^2(\nu) \rightarrow L^2(\nu^{-1})$ is the adjoint of $\Pi_b : L^2(\nu) \rightarrow L^2(\nu^{-1})$.

\vspace{0.1in}
\noindent $(4) \Rightarrow (1)$: Suppose $\Pi_b : L^2(\nu) \rightarrow L^2(\nu^{-1})$ is bounded. Then
	$$\|\Pi_b\unit_Q\|_{L^2(\nu^{-1})} \leq A\nu(Q)^{\frac{1}{2}} \:\: \text{ and } \:\:
		\|\Pi^*_b\unit_Q\|_{L^2(\nu^{-1})} \leq A\nu(Q)^{\frac{1}{2}}.$$
In this situation we know that both paraproducts are bounded, so we can get to the $BMO^2_{\mcd}(\nu)$ norm of $b$ faster than in the square function approach, by noting that
	\begin{equation} \label{E:BMOPiPiStar}
	\unit_Q (b - \La b\Ra_Q) = \unit_Q \left( \Pi_b \unit_Q - \Pi^*_b\unit_Q \right).
	\end{equation}
Then
	\begin{align*}
	\int_Q |b - \La b\Ra_Q|^2 \,d\nu^{-1} & = \int_Q \left| \Pi_b \unit_Q - \Pi^*_b\unit_Q\right|^2 \,d\nu^{-1}\\
		& \leq 2\|\Pi_b\unit_Q\|^2_{L^2(\nu^{-1})} + 2\|\Pi^*_b\unit_Q\|^2_{L^2(\nu^{-1})}\\
		& \leq 4A^2 \nu(Q),
	\end{align*}
which gives exactly $\|b\|_{BMO^2_{\mcd}(\nu)} \leq 2A$.
\end{proof}

We note here that the equivalence $(1) \Leftrightarrow (4)$ in fact holds for any $A_2$ weight. Moreover, the strategy used in proving this last equivalence above can be employed to give more precise bounds for the quantities $\mathbb{B}_1^{\mcd}(b,\mu,\lb)$ and $\mathbb{B}_2^{\mcd}(b,\mu',\lb')$ when $b\in BMO_D^2(\nu)$. For in this case, we know from Theorem \ref{T:ParaprodTwoWeight} that
	\begin{align*}
	\|\Pi_b\unit_Q\|_{L^p(\lb)} + \|\Pi^*_b\unit_Q\|_{L^p(\lb)}  & \lesssim  \|b\|_{BMO^2_{\mcd}(\nu)} \mu(Q)^{\frac{1}{p}},  
\\
	\|\Pi_b\unit_Q\|_{L^q(\mu')} + \|\Pi^*_b\unit_Q\|_{L^q(\mu')}&\lesssim  \|b\|_{BMO^2_{\mcd}(\nu)} \lb'(Q)^{\frac{1}{q}}. 
	\end{align*}
Then, using \eqref{E:BMOPiPiStar}, we have for any $Q\in\mcd$:
	\begin{align*}
	\left(\int_Q |b - \La b\Ra_Q|^p\,d\lb\right)^{\frac{1}{p}} & = \|\unit_Q (\Pi_b\unit_Q - \Pi_b^*\unit_Q)\|_{L^p(\lb)}\\
		& \leq \|\Pi_b\unit_Q\|_{L^p(\lb)} + \|\Pi^*_b\unit_Q\|_{L^p(\lb)}\\
		& \lesssim  \|b\|_{BMO^2_{\mcd}(\nu)} \mu(Q)^{\frac{1}{p}}.
	\end{align*}
The similar statement for $\mathbb{B}_2^{\mcd}(b, \mu',\lb')$ follows immediately by considering the paraproducts as operators $L^q(\lb') \rightarrow L^q(\mu')$. We state this result separately.

\begin{prop} \label{P:Pre-BloomLemma}
Let $\mu, \lb \in A_p$ with $1<p<\infty$ and put $\nu \defeq \mu^{\frac{1}{p}} \lb^{-\frac{1}{p}}$. Then for any dyadic grid $\mcd$ on $\R^n$ and any $b \in BMO^2_{\mcd}(\nu)$:
	\begin{equation} \label{E:Pre-BloomLemma1}
	\mathbb{B}_1^{\mcd}(b, \mu, \lb) \lesssim   \|b\|_{BMO^2_{\mcd}(\nu)} \text{, and } \mathbb{B}_2^{\mcd}(b, \mu', \lb') \lesssim   \|b\|_{BMO^2_{\mcd}(\nu)}.
	\end{equation}
\end{prop}

\section{Two-Weight Inequalities for Commutators with Calder\'on-Zygmund Operators}
\label{s:UpperBound}



We  prove Theorem \ref{T:UpperBound}, our upper bound on commutators. 
By the Hyt{\"o}nen Representation Theorem \ref{T:HytRep},
	\begin{equation} \label{E:UBP1}
	\La [b, T]f, g\Ra = c(n, T) \:\mathbb{E}_{\omega} \sum_{i,j=0}^{\infty} 2^{- \kappa _{i,j}\frac{\delta}{2}}
		\La [b, \mbs_{\omega}^{ij}]f, g\Ra,
	\end{equation}
for all bounded, compactly supported $f$, $g$, so it suffices to show that the commutators $[b, \mbs_{\omega}^{ij}]$ are bounded $L^p(\mu) \rightarrow L^p(\lb)$ uniformly in $\omega$.  We claim that for any choice of $\mcd^{\omega}$:
	\begin{equation} \label{E:SijUBd}
	\left\| [b, \mbs_{\omega}^{ij}] : L^p(\mu) \rightarrow L^p(\lb) \right\|   \lesssim  \kappa_{ij} \|b\|_{BMO(\nu)},
	\end{equation}
for all non-negative integers $i$, $j$, where recall that $ \kappa _{i,j}$ is defined in \eqref{e:kappa}, 
and in particular is at most linear in $ i+j$.  The linear growth in complexity in the second estimate is dominated 
by the exponential decay in complexity in the first.  Hence we conclude an upper bound on the norm of the 
commutator, completing the proof of Theorem~\ref{T:UpperBound}.

In what follows, consider $\mcd \defeq \mcd^{\omega}$ to be fixed and simply write $\mbs^{ij}$. The commutator 
$[b, \mbs^{ij}]f = b \mbs^{ij}f  - \mbs ^{ij} (bf)$.  Expand the products  into paraproducts as in \eqref{E:ParaprodDecomp} and obtain
\begin{gather*}
[b, \mbs^{ij}]f = T_1f + T_2f + \mcr^{ij}f,
\\
\textup{where} \qquad 
T_1f \defeq (\Pi_b + \Pi^*_b + \Gamma_b)(\mbs^{ij}f), \:\:\: T_2f \defeq \mbs^{ij}(\Pi_b + \Pi^*_b + \Gamma_b)f,
\\
\textup{and} \qquad 
\mcr^{ij}f \defeq \Pi_{\mbs^{ij}f}b - \mbs^{ij} \Pi_f b.
\end{gather*}
In this equality, the principal terms are $ T_1$ and $ T_2$.  
Using \eqref{E:SijWeightedUB} and Theorem \ref{T:ParaprodTwoWeight}, we easily obtain $\|T_k : L^p(\mu) \rightarrow L^p(\lb)\|  \lesssim  \kappa_{ij}
		\|b\|_{BMO(\nu)}$ for $k=1,2$.
So we only need to analyze the remainder term $\mcr^{ij}$. In what follows, we will show that
	\begin{equation} \label{E:UBP_RijUB}
	\left\| \mcr^{ij} : L^p(\mu) \rightarrow L^p(\lb)\right\|   \lesssim  \kappa_{ij} \|b\|_{BMO(\nu)},
	\end{equation}
for all $i$, $j$. Then \eqref{E:SijUBd}  follows.

\subsection{Remainder Estimate for $(i, j) \neq (0, 0)$}

In this case, the dyadic shift $\mbs^{ij}$ is cancellative:
	$$\mbs^{ij}f \defeq \sum_{\substack{R\in\mcd \\ \ep,\eta\not\equiv 1}} \sum_{\substack{P\in R_{(i)} \\ Q\in R_{(j)}}}
		a^{\ep\eta}_{PQR} \widehat{f}(P,\ep) h_Q^{\eta}.$$
Then for any $N \in \mcd$ and $\gamma\not\equiv 1$:
	$$\mbs^{ij}h_N^{\gamma} = \sum_{\eta\not\equiv 1} \sum_{Q\in(N^{(i)})_{(j)}} a^{\gamma\eta}_{NQN^{(i)}} h_Q^{\eta}
		\:\:\:\text{ and }\:\:\:
		\La \mbs^{ij}f, h_N^{\gamma}\Ra = \sum_{\ep\not\equiv 1} \sum_{P \in (N^{(j)})_{(i)}} a^{\ep\gamma}_{PNP^{(i)}} \widehat{f}(P,\ep).$$
These expressions give us the two terms in the remainder as
	\begin{equation} \label{E:PiSijfb}
	\Pi_{\mbs^{ij}f}b = \sum_{\substack{R\in\mcd \\ \ep,\eta\not\equiv 1}} \sum_{\substack{P\in R_{(i)} \\ Q\in R_{(j)}}}
		a^{\ep\eta}_{PQR} \widehat{f}(P,\ep)\La b\Ra_Q h_Q^{\eta}, \:\:\:\:\:\:\:
	\mbs^{ij}\Pi_fb = \sum_{\substack{R\in\mcd \\ \ep,\eta\not\equiv 1}} \sum_{\substack{P\in R_{(i)} \\ Q\in R_{(j)}}}
		a^{\ep\eta}_{PQR} \widehat{f}(P,\ep)\La b\Ra_P h_Q^{\eta}.
	\end{equation}
From \eqref{E:PiSijfb}:
	\begin{equation} \label{E:UBP_Rij}
	\mcr^{ij}f = \sum_{\substack{R\in\mcd \\ \ep,\eta\not\equiv 1}} \sum_{\substack{P\in R_{(i)} \\ Q\in R_{(j)}}}
		a^{\ep\eta}_{PQR} \widehat{f}(P,\ep)\left( \La b\Ra_Q - \La b\Ra_P\right) h_Q^{\eta}.
	\end{equation}
The difference in the averages of $ b$ is an essential term, but the cubes $ P$ and $ Q$ above are just descendants of $R$.
They need not intersect.  

\smallskip 

We continue our analysis of the remainder in terms of the relative sizes of $i$ and $j$, but the cases of $ i\leq j$ and 
$ j \leq i$ are dual, and so we only consider the former. 
Each $Q\in R_{(j)}$ is contained in a unique $N \in R_{(i)}$, and then $Q \in N_{(j-i)}$. (Note that $N = Q$ if $i=j$.) Rewrite $\mcr^{ij}f$ by grouping the $Q$'s this way:
	$$\mcr^{ij}f = \sum_{\substack{R\in\mcd\\ \ep,\eta\not\equiv 1}} \sum_{P, N \in R_{(i)}} \widehat{f}(P,\ep)
		\sum_{Q\in N_{(j-i)}} a_{PQR}^{\ep\eta} \big(\La b\Ra_{Q} - \La b\Ra_P\big) h_Q^{\eta},$$
and write
	$$\La b\Ra_{Q} - \La b\Ra_P = (\La b\Ra_{Q} - \La b\Ra_N) + \left(\La b\Ra_{N} - \La b\Ra_R\right) + 
		\left(\La b\Ra_{R} - \La b\Ra_P\right).$$
Note that the first term disappears if $i=j$, and also the expansion  \eqref{E:AvgDifference} applies to each of the terms in the parentheses above. The remainder is the sum of three terms. 
	\begin{align*}
	\mcr^{ij}f = & \sum_{k=1}^{j-i} \sum_{\substack{R\in\mcd \\ \ep,\eta,\gamma\not\equiv 1}} \sum_{\substack{P\in R_{(i)} \\ N\in R_{(j-k)}}}
		\widehat{f}(P,\ep) \widehat{b}(N,\gamma) \sum_{Q\in N_{(k)}} a_{PQR}^{\ep\eta} h_N^{\gamma} (Q) h_Q^{\eta}\\
		+ & \sum_{k=1}^{i} \sum_{\substack{R\in\mcd \\ \ep,\eta,\gamma\not\equiv 1}} \sum_{\substack{P\in R_{(i)} \\ N\in R_{(i-k)}}}
		\widehat{f}(P,\ep) \widehat{b}(N,\gamma) \sum_{Q\in N_{(j-i+k)}} a_{PQR}^{\ep\eta} h_N^{\gamma} (Q) h_Q^{\eta}\\
		- & \sum_{k=1}^i \sum_{\substack{R\in\mcd \\ \ep,\eta,\gamma\not\equiv 1}} \sum_{\substack{N\in R_{(i-k)} \\ Q\in R_{(j)}}}
			\left(\sum_{P\in N_{(k)}} \widehat{f}(P,\ep) h_N^{\gamma}(P)\right) \widehat{b}(N,\gamma) a^{\ep\eta}_{PQR} h_Q^{\eta}.
	\end{align*}
We relabel the second term by replacing $k$ with $k-j+i$, and then combine it with the first term. Finally, we may write
	\begin{equation} \label{E:UBP_RijAB}
	\mcr^{ij}f = \sum_{k=1}^j A_kf - \sum_{k=1}^i B_kf,
	\end{equation}
where:
	$$A_k f \defeq \sum_{\substack{R\in\mcd \\ \ep,\eta,\gamma\not\equiv 1}} \sum_{\substack{P\in R_{(i)} \\ N\in R_{(j-k)}}}
		\widehat{f}(P,\ep) \widehat{b}(N,\gamma) \sum_{Q\in N_{(k)}} a_{PQR}^{\ep\eta} h_N^{\gamma} (Q) h_Q^{\eta},$$
and
	$$B_k f \defeq \sum_{\substack{R\in\mcd \\ \ep,\eta,\gamma\not\equiv 1}} \sum_{\substack{N\in R_{(i-k)} \\ Q\in R_{(j)}}}
			\left(\sum_{P\in N_{(k)}} a^{\ep\eta}_{PQR} \widehat{f}(P,\ep) h_N^{\gamma}(P)\right) \widehat{b}(N,\gamma) h_Q^{\eta}.$$

It suffices to prove that $\left\Vert A_k: L^2(\mu)\to L^2(\lambda)\right\Vert+\left\Vert B_k: L^2(\mu)\to L^2(\lambda)\right\Vert  \lesssim  \|b\|_{BMO(\nu)}$, because then from \eqref{E:UBP_RijAB} we obtain:
$$
\left\Vert \mcr^{ij} : L^2(\mu)\to L^2(\lambda)\right\Vert \leq j \left\Vert A_k: L^2(\mu)\to L^2(\lambda)\right\Vert+ i\left\Vert B_k: L^2(\mu)\to L^2(\lambda)\right\Vert \lesssim  \kappa_{ij} \|b\|_{BMO(\nu)},
$$
which is nothing other than \eqref{E:UBP_RijUB}.  We now turn to computing the norms of $A_k$ and $B_k$.

We begin with $A_k$ and again proceed by duality. We let $f \in L^p(\mu)$ and $g\in L^q(\lb')$ and  
appeal to $ H ^{1}$-$BMO$ duality, as expressed in \eqref{E:NuH1BMODuality}, to get the $ BMO _{\nu }$ norm.  
	$$|\La A_kf, g\Ra| = |\La b, \Phi\Ra| \lesssim [\nu]_{A_2} \|b\|_{BMO^2_{\mcd}(\nu)} \|S_{\mcd}\Phi\|_{L^1(\nu)},$$
where, as before, $ \Phi $ is a bilinear expression involving $ f$ and $ g$. 
	$$\Phi \defeq \sum_{\substack{R\in\mcd \\ \ep,\eta,\gamma\not\equiv 1}} \sum_{\substack{P\in R_{(i)} \\ N\in R_{(j-k)}}} 
		\widehat{f}(P,\ep) \left(\sum_{Q\in N_{(k)}} a^{\ep\eta}_{PQR} h_N^{\gamma}(Q) \widehat{g}(Q,\eta)\right)h_N^{\gamma}.$$
This is a Haar series, and we pass to its square function, summing over the cubes $ N$.  
	$$(S_{\mcd}\Phi)^2 \lesssim 2^{-n(i+j)} \sum_{\substack{N\in\mcd\\ \ep,\eta\not\equiv 1}} \left( 
		\sum_{P\in(N^{(j-k)})_{(i)}} |\widehat{f}(P,\ep)| \sum_{Q\in N_{(k)}} \frac{1}{\sqrt{|N|}} |\widehat{g}(Q,\eta)| \right)^2 
		\frac{\unit_N}{|N|}.$$
The term $ 2 ^{-n (i+j)}$ comes from the decay of the Haar shift coefficients.  The sum involving $ g$ is bounded by 
	$$\sum_{Q\in N_{(k)}} \frac{1}{\sqrt{|N|}} |\widehat{g}(Q, \eta)| \leq \frac{2^{\frac{kn}{2}}}{|N|} 
		\sum_{Q\in N_{(k)}} \int_Q |g|\,dx = 2^{\frac{kn}{2}} \La |g|\Ra_N.$$
So the square function is bounded by 
	\begin{align*}
	(S_{\mcd} \Phi )^2 & \lesssim 2^{-n(i+j-k)} (Mg)^2 \sum_{N\in\mcd, \ep\not\equiv 1} \left( \sum_{P \in (N^{(j-k)})_{(i)}} |\widehat{f}(P,\ep)| \right)^2 \frac{\unit_N}{|N|}\\
		& = 2^{-n(i+j-k)} (Mg)^2 \left(\widetilde{S_{\mcd}}^{i,j-k}f\right)^2.
	\end{align*}
The maximal function is controlled by Muckenhoupt's bound, and the square function by the estimate \eqref{E:ShiftedDSFUB}. We have 
	\begin{align*}
	\|S_{\mcd}f\|_{L^1(\nu)} & \lesssim 2^{-\frac{n}{2}(i+j-k)} \int_{\R^n} (Mg) (\widetilde{S_{\mcd}}^{i,j-k}f) \,d\nu &\\
		& \leq 2^{-\frac{n}{2}(i+j-k)} \|Mg\|_{L^q(\lb')} \left\| \widetilde{S_{\mcd}}^{i,j-k}f \right\|_{L^p(\mu)} &\\
		& \lesssim 2^{-\frac{n}{2}(i+j-k)} \|g\|_{L^q(\lb')} 2^{\frac{n}{2}(i+j-k)} 
			\|f\|_{L^p(\mu)} 
 =  \|f\|_{L^p(\mu)} \|g\|_{L^q(\lb')}.  
	\end{align*}
The completes the proof of 
	$\|A_k : L^p(\mu) \rightarrow L^p(\lb)\|  \lesssim  \|b\|_{BMO(\nu)}$.

	\smallskip 

Similarly for $B_k$:
	$$|\La B_kf, g\Ra| = |\La b, \Phi\Ra| \lesssim  \|b\|_{BMO^2_{\mcd}(\nu)} \|S_{\mcd}\Phi\|_{L^1(\nu)},$$
where 
	$$\Phi \defeq \sum_{\substack{R\in\mcd \\ \ep,\eta,\gamma\not\equiv 1}} \sum_{\substack{N\in R_{(i-k)} \\ Q\in R_{(j)}}} 
	\left(\sum_{P\in N_{(k)}} \widehat{f}(P, \ep) h_N^{\gamma}(P) a_{PQR}^{\ep\eta}\right) \widehat{g}(Q, \eta) h_N^{\gamma}.$$
The analysis of the square function $ S _{\mathcal D} \Phi $ is symmetric with respect to the roles of $ f$ and $ g$. 
The proof is analogous, and so omitted.  


\subsection{Remainder Estimate for $i = j = 0$} \label{Ss:00Remainder}

A precise analysis of the case $i=j=0$ in Theorem \ref{T:HytRep} is given in \cite{HytRepOrig}, where it is shown that $\mbs^{00}$ is of the form
	$$\mbs^{00} = \mbs^{00}_c + \Pi_a + \Pi^*_d,$$
where $\mbs^{00}_c$ is a \textit{cancellative} dyadic shift with parameters $(0,0)$, and $\Pi_a$, $\Pi^*_d$ are paraproducts with symbols $a, d \in BMO_{\mcd}$ with 
	$\|a\|_{BMO_{\mcd}} \leq 1; \:\: \|d\|_{BMO_{\mcd}} \leq 1$.  
(The definition of the paraproduct is in \eqref{e:para}.)	
Here $BMO_{\mcd}$ denotes the \textit{unweighted} dyadic $BMO$ space. The functions $a$ and $d$ come from the T1 theorem of David-Journ{\'e}.  The remainder $\mcr^{00}$ then has the form
	$$\mcr^{00} = \mcr^{00}_c  + \mcr_a  + \mcr^*_d,$$
where
$$\mcr^{00}_c f \defeq \Pi_{\mbs_c^{00}f}b - \mbs_c^{00}\Pi_f b; \:\:\:\:\: \mcr_a f \defeq \Pi_{\Pi_a f}b - \Pi_a\Pi_f b; \:\:\:\:\:
	\mcr^*_d f \defeq \Pi_{\Pi^*_d f}b - \Pi^*_d\Pi_f b.$$
Now, $\mbs^{00}_c$ is cancellative, so
	$$\mbs^{00}_c f = \sum_{R\in\mcd}\sum_{\ep,\eta\not\equiv 1} a_R^{\ep\eta} \widehat{f}(R,\ep) h_R^{\eta},$$
for some $|a_R^{\ep\eta}| \leq 1$. It is easy to check that
	$$\Pi_{\mbs_c^{00}f}b = \mbs_c^{00}\Pi_f b = \sum_{R\in\mcd} \sum_{\ep,\eta\not\equiv 1}
		a_R^{\ep\eta} \widehat{f}(R,\ep)\La b\Ra_{R} h_R^{\eta}.$$
So the term $\mcr^{00}_c = 0$, and we only need to look at $\mcr_a$ and $\mcr^*_d$.

We recall the  $ A_p$ bounds for paraproduct operators, which  is classical. Namely, for 
$a\in BMO_{\mcd}$ and a weight $w\in A_p$ with $1<p<\infty$ we have:
	\begin{equation} \label{E:ParaprodBMOd}
	\left\|\Pi_a f\right\|_{L^p(w)} \lesssim  \|a\|_{BMO_{\mcd}} \|f\|_{L^p(w)}.
	\end{equation}
Let us look at the term $\mcr_a f$:
	$$\mcr_a f = \sum_{Q\in\mcd,\ep\not\equiv 1} \widehat{a}(Q,\ep) h_Q^{\ep}
		\sum_{R\supsetneq Q,\eta\not\equiv 1} \widehat{f}(R,\eta) \left[\La b\Ra_Q - \La b\Ra_R\right] h_R^{\eta}(Q).$$
We write
	$$\La b\Ra_Q - \La b\Ra_R = \sum_{\substack{N \in \mcd \\ Q \subsetneq N \subseteq R}} \sum_{\gamma\not\equiv 1} \widehat{b}(N,\gamma) h_N^{\gamma}(Q),$$
and express $\mcr_a f$ as a sum of three terms, called $A$, $B$, and $C$, which we analyze separately. Specifically, we look at the cases $N \subsetneq R$, $N = R$ with $\gamma\neq\eta$ and $N=R$ with $\gamma=\eta$. For the first case:
	\begin{align*}
	A & \defeq \sum_{Q\in\mcd,\ep\not\equiv 1} \widehat{a}(Q,\ep) h_Q^{\ep}
		\sum_{R\supsetneq Q,\eta\not\equiv 1} \sum_{Q \subsetneq N\subsetneq R} \sum_{\gamma\not\equiv 1}
		\widehat{b}(N,\gamma) \widehat{f}(R,\eta) h_N^{\gamma}(Q) h_R^{\eta}(Q) \\
	&= \sum_{Q\in\mcd,\ep\not\equiv 1} \widehat{a}(Q,\ep) h_Q^{\ep}
		 \sum_{N\supsetneq Q,\gamma\not\equiv 1} \widehat{b}(N,\gamma) \La f\Ra_N h_N^{\gamma}(Q) \\
	&= \sum_{Q\in\mcd,\ep\not\equiv 1} \widehat{a}(Q,\ep) \La\Pi_bf\Ra_Q h_Q^{\ep} = \Pi_a\Pi_bf.
	\end{align*}
The second case similarly gives:
	\begin{align*}
	B & \defeq \sum_{Q\in\mcd,\ep\not\equiv 1} \widehat{a}(Q,\ep) h_Q^{\ep}
		\sum_{R\supsetneq Q,\eta\not\equiv 1} \sum_{\gamma\not\equiv 1,\gamma\neq\eta}
		\widehat{b}(R,\gamma) \widehat{f}(R,\eta) h_R^{\gamma}(Q) h_R^{\eta}(Q)\\
	& = \sum_{Q\in\mcd,\ep\not\equiv 1} \widehat{a}(Q,\ep) \La \Gamma_bf\Ra_Q h_Q^{\ep} = \Pi_a\Gamma_bf.
	\end{align*}
Here, $ \Gamma _b $ is defined in \eqref{e:gamma}.  Finally, the case $N=R, \gamma=\eta$ yields:
	\begin{align}
	C & \defeq \sum_{Q\in\mcd,\ep\not\equiv 1} \widehat{a}(Q,\ep) h_Q^{\ep} 
		\sum_{R\supsetneq Q,\eta\not\equiv 1} \widehat{b}(R,\eta) \widehat{f}(R,\eta)	\frac{1}{|R|}\\
	& = \sum_{Q\in\mcd,\ep\not\equiv 1} \widehat{a}(Q,\ep) h_Q^{\ep} \left(
		\La\Pi_b^*f\Ra_Q  - \frac{1}{|Q|}\sum_{P\subseteq Q, \eta\not\equiv 1} \widehat{b}(P,\eta) \widehat{f}(P, \eta)	\right)\\
	& = \Pi_a\Pi_b^*f - \Lambda_{a,b}f,
\\  \textup{where} \qquad 
\label{E:Lambda_ab}
	\Lambda_{a, b}f &\defeq \sum_{Q\in\mcd,\ep\not\equiv 1} \widehat{a}(Q,\ep) \frac{1}{|Q|} \left(
	\sum_{P\subseteq Q,\eta\not\equiv 1} \widehat{b}(P,\eta) \widehat{f}(P, \eta) \right)  h_Q^{\ep}.
	\end{align}
In summary
	\begin{equation} \label{E:Raf}
	\mcr_a  = \Pi_a\Pi_b  + \Pi_a\Gamma_b  + \Pi_a \Pi^*_b  - \Lambda_{a, b}.
	\end{equation}
A similar analysis of $\mcr_d^*$ shows that
	\begin{align} \label{E:Rd*f}
	\mcr_d^* &= \Lambda_{d,b}^* - \Pi_b\Pi_d^* - \Gamma_b\Pi_d^* - \Pi^*_b\Pi_d^*,
\\  \textup{where} \qquad 
 \label{E:Lambda*_db}
	\Lambda_{d,b}^*f &\defeq \sum_{Q\in\mcd,\ep\not\equiv 1} \widehat{d}(Q,\ep) \widehat{f}(Q,\ep)\frac{1}{|Q|}
		\left(\sum_{P\subseteq Q, \eta\not\equiv 1} \widehat{b}(P,\eta)h_P^{\eta}\right).
\end{align}
We need to compute the $ L ^{p} (\nu )$-norm for the terms in \eqref{E:Raf} and \eqref{E:Rd*f}. 
For the terms that involve the composition of paraproducts and $ \Gamma $, first use the 
two-weight inequalities of Theorem~\ref{T:ParaprodTwoWeight}, and then \eqref{E:ParaprodBMOd}. 
It remains to show that $\Lambda_{a,b}$ is bounded.

\begin{lm}\label{L:LambdaBounds}
Let $a \in BMO_{\mcd}$ and $b \in BMO_{\mcd}^2(\nu)$. These inequalities hold. 
	\begin{align} \label{E:LambdaBound}
	\left\|\Lambda_{a,b} : L^p(\mu) \rightarrow L^p(\lb) \right\|& \lesssim 
\|a\|_{BMO_{\mcd}} \|b\|_{BMO^2_{\mcd}(\nu)},
\\ \label{E:Lambda*Bound}
	\left\|\Lambda^*_{a,b} : L^p(\mu) \rightarrow L^p(\lb) \right\|  &\lesssim 
	\|a\|_{BMO_{\mcd}} \|b\|_{BMO^2_{\mcd}(\nu)} . 
	\end{align}
\end{lm}

\begin{proof} The proof of the first set of inequalities is given, with the other set following by similar reasoning.  
We argue by duality, so that the appeal to weighted $ H ^{1}$-$BMO$ duality is easy.  For $ f\in L ^{p} (\mu )$, and $ g\in L ^{p'} (\lambda' )$, we have 
\begin{align*}
\langle  \Lambda _{a,b} f, g\rangle 
&=
\sum_{Q\in\mcd,\ep\not\equiv 1} \widehat{a}(Q,\ep) \frac{1}{|Q|}  
	\sum_{P\subseteq Q,\eta\not\equiv 1} \widehat{b}(P,\eta) \widehat{f}(P, \eta) \widehat g( Q, \epsilon )
\\
&= \sum_{P\in\mcd, \eta\not\equiv 1}  \widehat{b}(P,\eta)   \widehat{f}(P, \eta)  \Psi _P 
\\
\textup{where} \qquad 
\Psi _P &\defeq   
\sum_{Q \supseteq P, \ep\not\equiv 1} \widehat{a}(Q,\ep) \frac{1}{|Q|}  \widehat g( Q, \epsilon ). 
\end{align*}
Recall that if we multiply the coefficients $ \widehat{a}(Q,\ep)$ by choices of signs, we do not increase 
the $BMO$ norm of $ a$.  The same remark applies to $ b \in BMO (\nu )$.  Therefore, since $ f$ and $ g$ are 
fixed, we are free to assume that \emph{each individual summand above is non-negative.} This only requires 
that we modify the Haar wavelet expansions of $ b$ and $ a$ by choices of signs, but this fact is suppressed in the notation. 

The key fact that this gives us is a control of the terms $ \Psi _P$, namely 
\begin{equation*}
\Psi _P \leq \langle   \Pi _{a} ^{\ast}  g\rangle_P \leq \inf _{x\in P} M (\Pi _{a} ^{\ast} g) (x). 
\end{equation*}
Here, we are using the adjoint paraproduct of applied to $ g$, as defined in \eqref{e:para}.  
And, then, using the weighted $ H ^{1}$-$BMO$ duality, as expressed in \eqref{E:H1BMODual1}, we have
\begin{align*}
\left\vert \langle  \Lambda _{a,b} f,  \Phi \rangle \right\vert &\leq  
\|b\|_{BMO^2_{\mcd}(\nu)} \|S \Phi\|_{L^1(\nu )}, 
 \\  
  \textup{where} \qquad 
  \Phi & \defeq \sum_{P\in\mcd, \eta\not\equiv 1}  \widehat{f}(P, \eta)  \Psi _P  h ^{\epsilon } _{P}.  
\end{align*}
And, last of all, using the definition of $ \mu $ and H\"older's inequality, and weighted inequalities for the maximal 
function and paraproduct operators, 
\begin{align*}
\|S \Phi\|_{L^1(\nu )} & \lesssim   \int  S f  \cdot M (\Pi _{a} ^{\ast} g)   \; \mu ^{\frac 1p} \lambda ^{- \frac 1p} \; dx 
\\
& \leq \lVert S f\rVert _{L ^{p} (\mu )}  \lVert  M (\Pi _{a} ^{\ast} g) \rVert _{L ^{p'} (\lambda' )} 
\\
& \lesssim  \|a\|_{BMO_{\mcd}} \lVert f\rVert _{L ^{p} (\mu )} \lVert g\rVert _{L ^{p'} (\lambda ') }.  
\end{align*} 

\end{proof}

Now we may combine the results in \eqref{E:LambdaBound} and \eqref{E:Lambda*Bound} with the rest of the terms in \eqref{E:Raf} and \eqref{E:Rd*f}, which are controlled by \eqref{E:ParaprodBMOd} and Theorem \ref{T:ParaprodTwoWeight}, and obtain:
	$\|\mcr_a : L^p(\mu) \rightarrow L^p(\lb) \|  \lesssim  \|b\|_{BMO(\nu)}$,
and 
	$\|\mcr^*_d : L^p(\mu) \rightarrow L^p(\lb) \|  \lesssim  \|b\|_{BMO(\nu)}.$
Then \eqref{E:UBP_RijUB} for the non-cancellative case follows.


\subsection{Characterization of Bloom $BMO$ by Commutators with the Riesz Transforms}

In this section we prove Theorem \ref{T:LowerBound}.  Note that the first part follows directly from Theorem \ref{T:UpperBound}, and so it only remains to prove the lower bound.  Suppose 
	$$\left\|[b, R_i] : L^p(\mu) \rightarrow L^p(\lb)\right\| < \infty, \: i = 1,\ldots,n,$$
where $R_i$ are the Riesz transforms. Then, since $[b, R_iR_j] = [b, R_i]R_j + R_i[b, R_j]$ and $R_i : L^r(w) \rightarrow L^r(w)$ is bounded for all $i = 1, \ldots,n$ and all $w \in A_r$ with $1<r<\infty$ \cite{PetermichlRiesz}, we have that $[b, K]$ is bounded $L^p(\mu) \rightarrow L^p(\lb)$ for all $K$ that are polynomials in the Riesz transforms and the norm of this operator is at most $\displaystyle\sum_{i=1}^n \left\|[b, R_i] : L^p(\mu) \rightarrow L^p(\lb)\right\|$.

We employ the standard computation in \cite{CRW}.  Let $\{Y_k\}$ be an orthonormal basis for the space of spherical harmonics of degree $n$. Then
	$$\sum_{k} |Y_k(x)|^2 = c_n |x|^{2n},$$
and, by homogeneity,
	$$Y_k(x-y) = \sum_{|\alpha| + |\beta| = n} a_{\alpha\beta}^k x^{\alpha} y^{\beta},$$
where we are using standard multi-index notation. As shown in \cite{CRW}:
	\begin{equation} \label{E:CRW}
	|Q|\left| (b - \La b\Ra_Q)\unit_Q \right| (x) = \frac{1}{c_n} \sum_{k,\alpha,\beta} a_{\alpha\beta}^k x^{\alpha} \Gamma_Q(x)
		\left([b, R^{(k)}]y^{\beta}\unit_Q(y)\right)(x),
	\end{equation}
for all cubes $Q$ centered at the origin, where $\Gamma_Q \defeq \unit_Q \text{sgn}(b - \La b\Ra_Q)$ and $R^{(k)}$ is the polynomial in the Riesz transforms associated with $Y_k(x)|x|^{-n}$.  Note that, since $Q$ is centered at the origin
	\begin{equation}
	\label{E:reallyobviousestimate}
	|t^{\alpha}| \lesssim l(Q)^{|\alpha|},
	\end{equation}
for all $t\in Q$. Then from \eqref{E:CRW} and \eqref{E:reallyobviousestimate}:
	\begin{align*}
	|Q| \left(\int_Q |b - \La b\Ra_Q|^p\,d\lb\right)^{\frac{1}{p}} & \lesssim 
		\sum_{k,\alpha,\beta} l(Q)^{|\alpha|} \left\|[b, R^{(k)}]y^{\beta}\unit_Q(y)\right\|_{L^p(\lb)}\\
	&  \lesssim  \sum_{k, \alpha, \beta} l(Q)^{|\alpha|} \left\|[b, R^{(k)}] : L^p(\mu) \rightarrow L^p(\lb)\right\|
		\left\|y^{\beta}\unit_Q(y)\right\|_{L^p(\mu)}\\
	&  \lesssim  \sum_{\alpha,\beta} l(Q)^{|\alpha|} \left(\sum_{i=1}^n \left\|[b, R_i] : L^p(\mu) \rightarrow L^p(\lb)\right\|\right) l(Q)^{|\beta|} \mu(Q)^{\frac{1}{p}}\\
	&  \lesssim  |Q| \left(\sum_{i=1}^n \left\|[b, R_i] : L^p(\mu) \rightarrow L^p(\lb)\right\|\right) \mu(Q)^{\frac{1}{p}}.
	\end{align*}
Since the argument is translation-invariant, we may conclude that
	$$\left(\frac{1}{\mu(Q)} \int_Q |b - \La b\Ra_Q|^p\,d\lb\right)^{\frac{1}{p}}  \lesssim  \sum_{i=1}^n \left\|[b, R_i] : L^p(\mu) \rightarrow L^p(\lb)\right\|,$$
for all cubes $Q \subset \R^n$. Then
	$$\mathbb{B}_1^{\mcd}(b, \mu, \lb)  \lesssim  \sum_{i=1}^n \left\|[b, R_i] : L^p(\mu) \rightarrow L^p(\lb)\right\|$$
for all dyadic grids $\mcd$ on $\R^n$. By Theorem \ref{T:BloomEquiv},
	\begin{equation} \label{E:LBdPf1}
	b \in BMO_{\mcd}(\nu) \text{ with } \|b\|_{BMO_{\mcd}(\nu)}  \lesssim  \sum_{i=1}^n \left\|[b, R_i] : L^p(\mu) \rightarrow L^p(\lb)\right\|,
	\end{equation}
for all $\mcd$.

To see how this implies that $b \in BMO(\nu)$, recall that there exist $2^n$ dyadic grids $\mcd^{\omega}$ such that for any cube $Q \subset \R^n$ there is $Q^{\omega} \in \mcd^{\omega}$ such that
	$$Q \subset Q^{\omega} \text{ and } l(Q^{\omega}) \leq 6l(Q).$$
See the proof in \cite{HytSharpRH}*{Proof of Theorem 1.10}. 
Now from \eqref{E:LBdPf1}
	$$|\La b\Ra_{Q^{\omega}} - \La b\Ra_Q| \leq \frac{1}{|Q|} \int_Q |b - \La b\Ra_{Q^{\omega}}|\,dx  \lesssim  \frac{\nu(Q^{\omega})}{|Q|} \sum_{i=1}^n \left\|[b, R_i] : L^p(\mu) \rightarrow L^p(\lb)\right\|,$$
so
	$$\int_Q |b - \La b\Ra_Q|\,dx  \lesssim  \nu(Q^{\omega})\sum_{i=1}^n \left\|[b, R_i] : L^p(\mu) \rightarrow L^p(\lb)\right\|.$$
But, using the doubling property of $A_p$ weights
	$$\nu(Q^{\omega}) \leq [\nu]_{A_2} \left(\frac{|Q^{\omega}|}{|Q|}\right)^2\nu(Q)  \lesssim  \nu(Q),$$
hence
	$$\sup_{Q} \left(\frac{1}{\nu(Q)} \int_Q |b - \La b\Ra_Q|\,dx\right)  \lesssim  \sum_{i=1}^n \left\|[b, R_i] : L^p(\mu) \rightarrow L^p(\lb)\right\|,$$
which is exactly the lower bound in \eqref{e:Requivalence}.


\subsection{Proof of Corollary \ref{C:Factorization}}
The first part of the proof will use the following duality statement, which can be found in \cite{GCDuality}: $(H^1(\nu))^* \equiv BMO^2(\nu)$, in the sense that every element of $(H^1(\nu))^*$ is of the form
	$$\Lambda_b : H^1(\nu) \ni h \mapsto \Lambda_bh \defeq \int_{\mathbb{R}^n} b(x)h(x)\,dx,$$
for a unique $b \in BMO^2(\nu)$, with $\|b\|_{BMO^2(\nu)}  \lesssim  \|\Lambda_b\|$. In terms of maximal functions, the weighted Hardy space $H^1(\nu)$ is defined as follows: let $\varphi \in \mathcal{S}(\mathbb{R}^n)$ with $\int_{\mathbb{R}^n} \varphi(x)\,dx = 1$ and set $\varphi_r(x) \defeq r^{-n} \varphi(x/r)$ for $r>0$ and $x\in \mathbb{R}^n$. Then
	$$H^1(\nu) \defeq \left\{ f \in \mathcal{S}'(\mathbb{R}^n): f^* \in L^1(\nu)\right\} \text{, with } 
	\|f\|_{H^1(\nu)} \defeq \|f^*\|_{L^1(\nu)},$$
where $f^*(x) \defeq \sup_{r>0} |f\ast \varphi_r(x)|$ is the maximal function. There are many equivalent ways to define the weighted Hardy spaces -- in terms of the square function, or in terms of an atomic decomposition -- see \cites{GarciaCuerva, Stromberg}.

Now let $f \defeq g_1 (Tg_2) - (T^*g_1)g_2$. Then for $g_1 \in L^q(\lb')$, $g_2 \in L^p(\mu)$, and any $b \in BMO^2(\nu)$:
	\begin{align}
	\left|\int_{\mathbb{R}^n} b(x)f(x)\,dx\right| & = \left| \int_{\mathbb{R}^n} g_1(x) [b, T]g_2(x)\,dx \right| & \label{E:Cor}\\
		& \leq \left\|g_1\right\|_{L^q(\lb')} \left\|[b, T]g_2\right\|_{L^p(\lb)} & \nonumber \\
		&  \lesssim  \|g_1\|_{L^q(\lb')} \|g_2\|_{L^p(\mu)} \|b\|_{BMO^2(\nu)} & \text{ by Theorem \ref{T:UpperBound}} \nonumber.
	\end{align}
Then clearly $f \in H^1(\nu)$, with $\|f\|_{H^1(\nu)}  \lesssim  \|g_1\|_{L^q(\lb')} \|g_2\|_{L^p(\mu)}$.

The second statement follows identically as in \cite{CRW}, with the appropriate modifications.  We consider the Banach space of functions $f \in L^1(\nu)$ which admit a decomposition as in \eqref{E:Decomposition}, normed by 
	$$\vert\vert\vert f\vert\vert\vert_{H^1(\nu)} \defeq \inf \left\{\sum_{i=1}^{n}\sum_{j=1}^{\infty} \|g_j^i\|_{L^q(\lb')} \|h_j^i\|_{L^p(\mu)}\right\},$$
where the infimum is over all possible decompositions of $f$. Part one of this corollary shows that this is a subspace of $H^1(\nu)$. Now \eqref{E:Cor} and Theorem \ref{T:LowerBound} show that
	$$\sup \left\{\left|\int_{\mathbb{R}^n} b(x)f(x)\,dx\right|: \vert\vert\vert f\vert\vert\vert_{H^1(\nu)} = 1\right\} \approx \|b\|_{BMO^2(\nu)},$$
which implies the norms $\vert\vert\vert \cdot\vert\vert\vert_{H^1(\nu)}$ and $\|\cdot\|_{H^1(\nu)}$ are equivalent (see \cite{CLMS} for the simple functional analysis argument that yields this). This completes the proof.



\end{document}